\documentclass{amsart}[12pt]
\usepackage{graphicx,graphics, amsmath, amsthm, amscd, amsfonts}

\usepackage{latexsym}
\makeatletter \oddsidemargin.9375in \evensidemargin \oddsidemargin
\newtheorem{theorem}{Theorem}[section]
\newtheorem{lemma}[theorem]{Lemma}

\theoremstyle{definition}
\newtheorem{definition}[theorem]{Definition}
\newtheorem{example}[theorem]{Example}

\newtheorem{remark}[theorem]{Remark}
\numberwithin{equation}{section}

\begin{document}
\Large
\title[Hyperinner Product Spaces II]
{Hyperinner Product Spaces II}

\author[Ali Taghavi and Roja Hosseinzadeh]{Ali Taghavi$^1$ and Roja Hosseinzadeh$^2$}

\address{{ $^{1,2}$Department of Mathematics, Faculty of Mathematical Sciences,
 University of Mazandaran, P. O. Box 47416-1468, Babolsar, Iran.}}

\email{$^{1}$Taghavi@umz.ac.ir, $^{2}$ro.hosseinzadeh@umz.ac.ir}



\begin{abstract}\large
In this paper, we extend the definition of hyperinner product defined on weak hypervector spaces with a hyperoperation scalar product to weak hypervector spaces with the hyperoperations sum and scalar products.
\end{abstract} \maketitle

\section{\textbf{Introduction}}

\noindent

 In 1988, the concept of hypervector space was
first introduced by Scafati-Tallini.  She later considered more properties of such spaces and in $[15]$ introduced the concept of weak hypervector spaces. Authors, in $[8-12]$ considered hypervector spaces in viewpoint of analysis. In mentioned papers, authors introduced the concepts such as dimension of hypervector spaces, normed hypervector spaces, operator on these spaces and another important concepts. Authors in $[7]$ introduced the definition of real inner product on a hypervector space with hyperoperations sum and scalar product over the real field. Moreover, in $[12]$ authors introduced the concept of inner product on a weak hypervector space which in this paper, we extend this concept to weak hypervector spaces with the hyperoperations sum and scalar products. This is our motivation to introduce a new definition of hypergroups. In this field, for other definitions and applications see $[1-6]$ and $[16]$.
\section{\textbf{Hypergroup}}
In this section we introduce a new definition of hypergroups.
\begin{definition}
 $[16]$ A hyperoperation over a nonempty set $X$ is a mapping of
$X \times X$ into the set of all nonempty subsets of $X$.
\label{2.1}
\end{definition}
\begin{definition}
$[16]$ A nonempty set $X$ with exactly one hyperoperation $ \sharp $ is
called a hypergroupoid.
Let $(X, \sharp)$ be a hypergroupoid. For every point $x \in X$ and every nonempty
subset $A$ of $X$, we defined $x \sharp A = \cup _{a \in A} x \sharp a$.
\label{2.2}
\end{definition}
\begin{definition}
 A hypergroupoid $(X, \sharp)$ is said to be a hypergroup if \\
1- $(x \sharp y) \sharp z \cap x \sharp (y \sharp z) \neq \emptyset $ for all $x,y,z \in X$;\\
2- There is an unique element $0 \in X$ such that for every $x \in X$ there exists an unique element $y \in X$ for
which $0 \in x \sharp y$ and $0 \in y \sharp x$. Here $y$ is denoted by $-x$;\\
3- $x \in ((x \sharp y) \sharp -y) \cap (x \sharp 0)$ for all $x,y \in X$.\\
$(X, \sharp)$ is said to be a commutative hypergroup when $x \sharp y=y \sharp x$ for all $x,y \in X$.
\label{2.3}
\end{definition}
\begin{example}
Let $\alpha$ be a real number with $ 0 < \alpha < 1$ and for $x,y \in \mathbb{R}^2$ define
$$ x \sharp y=\{re^{i \theta}: \alpha |x+y| < r < |x+y|, \theta = arg (x+y)\};$$
i.e. the set of all points in $\mathbb{R}^2$ belonging to the line segment whose vertices are the $\alpha (x+y)$ and $x+y$.
It's easy to see that $(\mathbb{R}^2, \sharp )$ is a hypergroup.
\label{2.4}
\end{example}
\begin{example}
Let $\alpha$ be a real number with $ 0 < \alpha < 1$ and for $x,y \in \mathbb{R}^2$ define
$$ x \sharp y=\{re^{i \theta}: \alpha |x+y| < r < |x+y|, 0 \leq \theta \leq 2 \pi \}.$$
It,s easy to see that $(\mathbb{R}^2, \sharp )$ is a hypergroup.
\label{2.5}
\end{example}
\begin{lemma}
Let $(X, \sharp)$ be a hypergroup. Then for every $x, y \in X$, there exists an element $ e \in x \sharp y$ such that $x \in e \sharp -y$.
\label{2.6}
\end{lemma}
\begin{proof}
It is easily follows from part 3 of Definition 2.3.
\end{proof}
\begin{definition} Let $(X, \sharp)$ be a hypergroup
and $x,y\in X$. Essential point of $x \sharp y$, that we denote it by $e_{x \sharp y}$, if $0 \not \in x \sharp y$, is the element of
$x \sharp y$ such that $x \in e_{x \sharp y} \sharp -y$. If $0 \in x \sharp y$, we define $e_{x \sharp y}=0$.
\label{2.7}
\end{definition}
\begin{example}
$e_{x \sharp y}$ in Example 2.4, for every $x,y \in \mathbb{R}^2$, is unique and equals to $x+y$, but $e_{x \sharp y}$ in Example 2.5 is not unique. For example for $x=1$ and $y=i$, two numbers $1+i$ and $-1+i$ are $e_{x \sharp y}$.
\label{2.8}
\end{example}
\begin{remark}
By the above example, $e_{x \sharp y}$ is not unique, necessarily. Hence we
denote the set of all essential points of $x \sharp y$ by $E_{x \sharp y}$.
\label{2.9}
\end{remark}
\section{\textbf{Hyperinner product spaces}}
In this section we introduce a new definition of a weak hypervector space and then define inner product on it.
\begin{definition}
 A $weak~hypervector~space$ over
a field $F$ is a quadruple $(X,\sharp,o,F)$ such that $(X, \sharp)$ is a commutative hypergroup and $o:F\times X\longrightarrow P_*(X)$ is a
multivalued product such that\\
1- $[ao(x \sharp y)] \cap [aox \sharp aoy]\neq
\emptyset$ for all $ a\in F$ and $x,y \in X$;\\
2- $[(a \sharp b)ox] \cap [aox \sharp box]\neq \emptyset$ for all $ a,b\in F$ and $x \in X$;\\
3- $ao(box)=(ab)ox$ for all $ a,b\in F$ and $x \in X$;\\
4- $ao(-x)=(-a)ox=-(aox)$ for all $ a\in F$ and $x \in X$;\\
5- $x \in 1ox$ for all $x \in X$;\\
The properties 1 and 2 are called weak right and left distributive
laws, respectively. Note that the set $ao(box)$ in 3 is of the
form $\cup_{y\in box}aoy$.
\label{3.1}
\end{definition}
For a hypervector space defined in the above definition, we can have a result similar to Lemma 2.6 in $[8]$ which implies the following definition.
\begin{definition}
 Let $X$ be a weak hypervector space over $F$, $a\in F$
and $x\in X$. Essential point of $aox$, that we denote it by $e_{aox}$, for $a\neq 0$ is the element of
$aox$ such that $x\in a ^{-1}oe_{aox}$. For $a =0$, we define
$e_{aox}=0$.
\label{3.2}
\end{definition}
\begin{remark}
Similar to $e_{x \sharp y}$, by the stated examples in $[8]$, $e_{aox}$ is not unique, necessarily.
 Hence we denote the set of all these elements by
$E_{aox}$.
\label{3.3}
\end{remark}
In the following first we introduce a new definition of norm on hypervector spaces defined in 3.1 and then inner products on such spaces, over real or complex fields.

Next assume that $X$ is a weak hypervector space over the field $F$, where the field $F$ is a real or complex field.
\begin{definition}
 A norm on $X$ is a mapping $\|.\| :X \longrightarrow
\mathbb{R}$ such that for every $a\in F$ and $x,y\in X$ we have

1- $\|x\|=0 \Leftrightarrow x=0$;

2- $  \mathrm{sup}\|x \sharp y\|\leq \|x\|+\|y\| $;

3- $ \mathrm{sup} \|aox\|=\vert a \vert\|x\| $.

 $(X, \|.\|)$ is called a normed weak hypervector space.
  \label{3.4}
 \end{definition}
\begin{lemma}
Let $ \|.\|$ be a norm on $X$. Then the following properties are hold.

1-  $  $  $\|-y\|= \|y\|$.

2-  $  $  $\mathrm{sup}\|x  \sharp 0\|= \|x\|$.

3-  $  $  $| \|x\|-\|y\| | \leq \mathrm{sup}\|x \sharp -y\|$.
\label{3.5}
\end{lemma}
\begin{proof}
1- We have $ -y \in 1o(-y)$ and by part 5 of Definition 3.1, $ (-1)oy=1o(-y)$. Thus by part 3 of Definition 3.2,
$$ \| -y  \| \leq  \mathrm{sup} \| (-1)oy\|= \| y  \|.$$
Similarly, we can conclude that $ \| y  \| \leq \| -y  \|$ and hence the proof is completed.

2- By part 3 of Definition 3.4, $\mathrm{sup}\|x  \sharp 0\| \leq \|x\|$. On the other hand, since $x \in x  \sharp 0$, we obtain $\|x\| \leq \mathrm{sup}\|x  \sharp 0\|$. Therefore, $\mathrm{sup}\|x  \sharp 0\|= \|x\|$.

3- By Definition 2.7, $e_{x \sharp -y} \in x \sharp -y$ and $x \in e_{x \sharp -y} \sharp -(-y)=e_{x \sharp -y} \sharp y$ which follow
\begin{eqnarray*}ý
 \begin{array}{lcl}ý
  \|x\| \leq \mathrm{sup} \|e_{x \sharp -y} \sharp y \| \leq  \|e_{x \sharp -y} \| + \| y \|\\
\Rightarrow  \|x\|-\|y\| \leq  \|e_{x \sharp -y} \| \leq  \mathrm{sup} \|x \sharp -y\|.
 ý\end{array}
 ý\end{eqnarray*}
 In a similar way we obtain $\|y\|- \|x\| \leq  \mathrm{sup} \|y \sharp -x\|= \mathrm{sup} \|x \sharp -y\|$ and so $| \|x\|-\|y\| | \leq \mathrm{sup}\|x \sharp -y\|$.
\end{proof}
\begin{definition}
Let $X$ be a normed weak hypervector space. An open hyperball is defined as $B_r(x) =\{ y \in X : \mathrm{sup}\|y-x\| < r\}$ with radius $r > 0$ and center at $x$.
\label{3.6}
\end{definition}
Authors in $[7]$ introduced the definition of inner product on a hypervector space with hyperoperations sum and scalar product over the real field. The following lemma shows that by the defined inner product in $[7]$, the set $x \sharp y$ is singleton. In fact, the operation sum is classic in the defined hyperinner product spaces in $[7]$.
\begin{lemma}
Let $X$ be the hypervector space over $\mathbb{R}$ as defined in $[7]$ and $(.,.)$ be also the defined inner product in $[7]$ on $X$. Then $x \sharp y$ is singleton for every $x,y \in X$.
\label{3.7}
\end{lemma}
\begin{proof}
First we assert that $(x \sharp y, z)$ is singleton for every $x,y,z \in X$.
For a commutative hypergroup in $[7]$, $-y \sharp y=\{0\}$ and by the defined inner product in it we have
$$ \mathrm{sup}(x \sharp y,z)= (x,z)+ (y,z).$$
Thus, $ \mathrm{sup}(-y \sharp y,z)= (-y,z)+ (y,z)$. On the other hand, $ \mathrm{sup}(-y \sharp y,z)= \mathrm{sup}(0,z)=0$. Therefore, $ (-y,z)=- (y,z)$.

Now for any $u$ belonging to $x \sharp y$,
$$ (u,z) \leq (x,z)+ (y,z).~~~\leqno(1)$$
By part 3 of definition of a hypergroup in $[7]$, $u \in x \sharp y$ yields that $x \in u \sharp -y$ and hence $ (x,z) \leq (u,z)+ (-y,z)$ and
$$ (x,z) \leq (u,z)-(y,z).~~~\leqno(2)$$
$(1)$ and $(2)$ yield that $ (u,z)= (x,z)+ (y,z)=  \mathrm{sup}(x \sharp y,z)$ and the proof of assertion is completed.

Let $u,v \in x \sharp y$. For any $z \in X$, by assertion we have $ (u,z)= (x,z)+ (y,z)$ and $ (v,z)= (x,z)+ (y,z)$. This for any $z \in X$ implies that
$$ (u-v,z)= (u,z)-(v,z)=0$$
and hence $(t,z)=0$ for any $t \in u-v$ and $z \in X$. Setting $z=t$ in the last relation implies that $(t,t)=0$ for any $t \in u-v$ and so $u-v=\{0\}$. Therefore, $u=v$ and then $x \sharp y$ is singleton.
\end{proof}
\begin{definition}
 An inner product on $X$ is a mapping $(.,.):X\times X
\rightarrow F$ such that for every $a\in F$ and $x,y,z\in X$ we have

 1-  $  $  $ (x,x)>0$ for $x\neq 0$;

 2-  $  $  $(x,x)=0 \Leftrightarrow x=0$;

 3-  $  $  $ \exists e_{x \sharp y} \in E_{x \sharp y}: \forall z \in X,~(x,z)+(y,z)=(e_{x \sharp y},z)$;

 4-  $  $  $ (y,x)=\overline{(x,y)}$;

 5-  $  $  $ \exists e_{aox} \in E_{aox}: \forall y \in X,~a(x,y)=(e_{aox},y)$;

 6-  $  $  $ (u,u) \leq (e_{x \sharp y},e_{x \sharp y})$ for every $u \in x \sharp y$;

 7-  $  $  $ (u,u) \leq (x,x)$ for every $u \in 1ox$.

 $X$ with an inner product is called a hyperinner product space.
 \label{3.8}
\end{definition}
Next assume that $X$ is a hyperinner product space over the field $F$ and
$e_{x \sharp y}$ and $e_{aox}$ is the same vectors in parts 3 and 5 of above definition, respectively.
\begin{lemma}
The following statements are hold for every $x,y \in X$ and $ a,b\in F$.

1-   $  $  $(0,x)=(x,0)=0$.

2-   $  $ $(-x,y)=(x,-y)=-(x,y)$.

3-   $  $ $(x,e_{aoy})=\overline{a}(x,y)$.

4-   $  $ $(u,u) \leq a^2(x,x)$ for every $u \in aox$.

5-   $  $ $a(box,y)=(abox,y)$.
\label{3.9}
\end{lemma}
\begin{proof}
1- Since $ e_{0 \sharp 0}=0$, by part 3 of Definition 3.8, $(0,x)+(0,x)=(0,x)$ and so $(0,x)=0$.

2- Since $ e_{x \sharp -x}=0$, by part 3 of Definition 3.8 and Part 1, $(x,y)+(-x,y)=(0,y)=0$ and so $-(x,y)=(-x,y)$.\\
For the second equality we have $(x,-y)=\overline{(-y,x)}=\overline{-(y,x)}=-(x,y)$.

3- $  $$(x,e_{aoy})=\overline{(e_{aoy},x)}=\overline{a(y,x)}=\overline{a}(x,y)$.

4- For every $u \in aox$ we have $e_{a^{-1}ou} \in 1ox$ and so $(e_{a^{-1}ou},e_{a^{-1}ou}) \leq (x,x)$. This by Part 3 implies that $ \frac{1}{a^2}(u,u) \leq (x,x)$.

5- By part 5 of Definition of 3.8
$$a(box,y)=\cup_{z\in box}a(z,y)  \subseteq \cup_{z\in box} (E_{aoz},y) \subseteq \cup_{z\in box} (abox,y)=(abox,y)$$
which implies
$$a^{-1}(abox,y) \subseteq (a^{-1}abox,y)=(box,y)$$
and this completes the proof.
\end{proof}
The following lemma shows that $\mathrm{sup}\|x  \sharp y\| \in  \|x  \sharp y\|$.
\begin{lemma}
 If $\|x\|=\sqrt{(x,x)}$, then the following statements are hold.

1-   $  $  $ |(x,y)| \leq \|x\| \|y\|$.

2-   $  $  $\mathrm{sup}\|x  \sharp y\| =  \|e_{x  \sharp y}\|$.
\label{3.10}
\end{lemma}
\begin{proof}
1- Set $A=\|x\|$, $B=|(x,y)|$ and $C=\|y\|$. There exists a complex number $ \alpha$ such that $ | \alpha |=1$ and $\alpha(y,x)=1$. Then for every real number $r$ we have
\begin{eqnarray*}
(e_{x \sharp -e_{r \alpha oy}},e_{x \sharp -e_{r \alpha oy}})&=&e(_{x \sharp -e_{r \alpha oy}}, x)-(e_{x \sharp -e_{r \alpha oy}},e_{r \alpha oy})\\
&=& (x, x)-(e_{r \alpha oy}, x)-(x,e_{r \alpha oy})-(e_{r \alpha oy},e_{r \alpha oy})\\
&=& (x, x)-r \alpha (y, x)-r \overline{ \alpha }(x,y)-r^2(y,y).
\end{eqnarray*}
Therefore for every real number $r$
$$Cr^2-2Br+A \geq 0.$$
If $C=0$, then $B=0$. Otherwise, we obtain $B^2 \leq AC$ and this completes the proof.

2- Let $u \in x \sharp y$. By part 6 of Definition 3.8, for every $u \in x \sharp y$ we have $(u,u) \leq (e_{x \sharp y},e_{x \sharp y})$ which implies that $\|u\| \leq  \|e_{x  \sharp y}\|$ and hence $\mathrm{sup}\|x  \sharp y\| \leq  \|e_{x  \sharp y}\|$. This together with $ e_{x  \sharp y} \in x  \sharp y$ yields the assertion.
\end{proof}
\begin{theorem}
 $\|x\|=\sqrt{(x,x)}$ has the properties of a norm on $X$.
\label{3.11}
\end{theorem}
\begin{proof}
The first and third properties of norm stated in Definition 3.4 can be proved in a similar way of the proof of Lemma 2.12 in $[12]$. We prove the second property. Let $u \in x \sharp y$. By part 6 of Definition 3.8 we have $(u,u) \leq (e_{x \sharp y},e_{x \sharp y})$ and by part 3
\begin{eqnarray*}
(e_{x \sharp y},e_{x \sharp y})&=& (x,e_{x \sharp y})+(y,e_{x \sharp y}) \\
 &=&(x,x)+(x,y)+(y,x)+(y,y)\\
 &=& \|x\|^2+  \|y\|^2+ 2 Re (x,y)\\
 & \leq& \|x\|^2+  \|y\|^2+ 2 \|x\| \|y\|\\
  &=&(\|x\|+  \|y\|)^2
\end{eqnarray*}
which implies that $\|u\| \leq \|x\|+  \|y\|$ for every $u \in x \sharp y$ and so $ \mathrm{sup}\|x \sharp y\|\leq \|x\|+\|y\| $.
\end{proof}
\par \vspace{.4cm}{\bf Acknowledgements.} This research is partially
Supported by the Research Center in Algebraic Hyperstructures and
Fuzzy Mathematics, University of Mazandaran, Babolsar, Iran.
\bibliographystyle{amsplain}

\begin{thebibliography}{10}\large

\bibitem{ta}  R. Ameri, M.M. Zahedi,
\textit{Hyperalgebraic systems}, Italian J. of
Pure and Appl. Math., No. 6 (1999).

\bibitem{ta}  J. Chvalina, S. Hoskova,
\textit{Modelling of join spaces with proximities by first-order linear partial
differential operators}, Italian J. of
Pure and Appl. Math., 21 (2007), 177-190.

\bibitem{ta}  P. Corsini,
\textit{Prolegomena of Hypergroup Theory}, Aviani Editore, 1993.

\bibitem{ta}  P. Corsini, V. Leoreanu,
\textit{Applications of Hypergroup Theory}, Kluwer Academic Publ., 2003.

\bibitem{ta}  B. Davvaz,
\textit{A brief survey of the theory of $H_v$-structures}, 8th AHA Congress, Spanidis Press
(2003), 39-70.


\bibitem{ta}  B. Davvaz, V. Leoreanu,
\textit{ Hyperring Theory and Applications}, Int. Academic Press, 2007.


 \bibitem{ma}  S. Roy, T.K. Samanta, \textit{Innerproduct hyperspaces}, arXiv:1009.3483v2 [math.GM] 5 Jun
2011.

\bibitem{ta}  A. Taghavi, R. Hosseinzadeh,
\textit{A note on dimension of weak hypervector spaces}, Italian J. of
Pure and Appl. Math., No. 33 (2014) 7-14.


\bibitem{ta}  A. Taghavi, R. Hosseinzadeh,
\textit{Operators on normed hypervector spaces}, Southeast Asian
Bulletin of Mathematics, (2011) 367-372.

\bibitem{ta}  A. Taghavi, R. Hosseinzadeh,
\textit{Operators on weak hypervector spaces}, Ratio Mathematica, 22 (2012) 37-43.

\bibitem{ta}  A. Taghavi, R. Hosseinzadeh,
\textit{Uniform Boundedness Principle for operators on
hypervector spaces}, Iranian Journal of Mathematical Sciences and Informatics,
Vol. 7, No. 2 (2012) 9-16.

\bibitem{ta}  A. Taghavi, R. Hosseinzadeh, H. Rohi,
\textit{Hyperinner product spaces}, Jornal of Algebraic Hyperstructures and it's Applications, Vol. 1, No.1 (2014) 95-100.

\bibitem{ta}  A. Taghavi, T. Vougiouklis, R. Hosseinzadeh,
\textit{A note on Operators on Normed Finite Dimensional Weak Hypervector
Spaces }, Scientific bulletin, Series A, Vol. 74, Iss. 4 (2012) 103-108.

\bibitem{ta}  M. Scafati-Tallini,
\textit{Characterization of remarkable Hypervector space}, Proc.
$8^{th}$ congress on "Algebraic Hyperstructures and Aplications",
Samotraki, Greece, (2002), Spanidis Press, Xanthi, (2003),
231-237.


\bibitem{ta}  M. Scafati-Tallini, \textit{Weak Hypervector space and
norms in such spaces}, Algebraic Hyperstructures and Applications
Hadronic Press. (1994), 199--206.

\bibitem{ta} T. Vougiouklis, \textit{From $H_v$-Rings to $H_v$-Fields}, Int. Journal of Algebraic Hyperstructures
and its Applications, 1 (1) (2014) 1-13.

\bibitem{vo2}  M. M. Zahedi, \textit{ A review on hyper k-algebras}, Iranian Journal of Mathematical Sciences and Informatics, Vol. 1, No. 1 (2006), 55-112.


\end{thebibliography}

\end{document}